\documentclass[12pt]{amsart}
\usepackage{amsmath}
\usepackage{amsfonts}
\usepackage{amsthm}
\usepackage{amssymb}
\usepackage{amscd}
\usepackage[all]{xy}
\usepackage{enumerate}
\usepackage{hyperref}
\usepackage{mathrsfs}
\usepackage{enumitem}


\subjclass[2010]{}

\newcommand*{\ext}{\mathcal{E}\kern -.7pt xt}
\theoremstyle{plain}
\newtheorem{thm}{Theorem}[section]

\newtheorem{prop}[thm]{Proposition}

\newtheorem{cor}[thm]{Corollary}

\theoremstyle{definition}
\newtheorem{defn}[thm]{Definition}

\newtheorem*{ackn}{Acknowledgment}

\newtheorem{rmk}[thm]{Remark}
\newcommand{\sA}{\mathcal{A}}

\newcommand{\sC}{\mathcal{C}}

\newcommand{\sE}{\mathcal{E}}

\newcommand{\sO}{\mathcal{O}}
\newcommand{\sP}{\mathcal{P}}
\newcommand{\sQ}{\mathcal{Q}}

\newcommand{\sU}{\mathcal{U}}

\newcommand{\sW}{\mathcal{W}}

\newcommand{\mC}{\mathbb{C}}
\newcommand{\mD}{\mathbb{D}}

\newcommand{\bu}{\mathbf{u}}
\newcommand{\Ima}{\mathrm{Im}\,}

\usepackage{color}
\numberwithin{equation}{section}

\newcommand{\beba}  {\begin{equation}\begin{array}{rcl}}

\newcommand{\eaee}  {\end{array}\end{equation}}

\makeatletter
\def\l@section{\@tocline{1}{0pt}{1pc}{}{}}
\def\l@subsection{\@tocline{2}{0pt}{1pc}{4.6em}{}}
\def\l@subsubsection{\@tocline{3}{0pt}{1pc}{7.6em}{}}
\renewcommand{\tocsection}[3]{%
  \indentlabel{\@ifnotempty{#2}{\makebox[2.3em][l]{%
    \ignorespaces#1 #2.\hfill}}}#3}
\renewcommand{\tocsubsection}[3]{%
  \indentlabel{\@ifnotempty{#2}{\hspace*{2.3em}\makebox[2.3em][l]{%
    \ignorespaces#1 #2.\hfill}}}#3}
\renewcommand{\tocsubsubsection}[3]{%
  \indentlabel{\@ifnotempty{#2}{\hspace*{4.6em}\makebox[3em][l]{%
    \ignorespaces#1 #2.\hfill}}}#3}
\makeatother

\title{Curvature and relative volume forms}

\author{Luca Rizzi}
\address{Luca Rizzi\\Center for Geometry and Physics \\
	Institute for Basic Science (IBS)\\
	Pohang 37673\\ Korea,
	\texttt{lucarizzi@ibs.re.kr}}

\author{Francesco Zucconi}
\address{Francesco Zucconi\\Department of Mathematics, Computer Science and Physics \\
Universit\`a degli Studi di Udine\\
Udine, 33100\\ Italia
\texttt{Francesco.Zucconi@uniud.it}}

\begin{document}

\markboth{}{}

\begin{abstract}
	Using metric techniques introduced by Berndtsson, we show a result on constancy of families dominated by a constant variety and, on the opposite side, a result on the strong non isotriviality of certain families of surfaces with positive index. We also give metric interpretations of liftability of relative volume forms and of strong non isotriviality in terms of the complex conjugate of a suitable  representative of the Kodaira-Spencer class.
\end{abstract}
\maketitle
%\tableofcontents

\section{Introduction}
In this paper we are concerned with the curvature properties of the higher direct images $R^if_*(\Omega^{n-i}_{X/B}(\log)\otimes L)$ and in particular of  $f_*(\omega_{X/B}\otimes L)$ where $f\colon X\to B$ is a semistable fibration of smooth complex projective varieties of relative
dimension $n$, $L$ is a Hermitian line bundle on $X$ and $\Omega^{\bullet}_{X/B}(\log)$ are the bundles of holomorphic relative logarithmic forms; see Section \ref{sez1} for the details. 

If $f\colon X\to B$ is smooth, it is classically known that $f_*\omega_{X/B}$ is Griffiths semi-positive, see  \cite{G2,gt}. If $B$ is a curve and $f$ is not necessarily smooth, a result by Fujita shows that $f_*\omega_{X/B}$ has a holomorphic decomposition $f_*\omega_{X/B}=\sU\oplus\sA$ where $\sA$ is an ample vector bundle and $\sU$ is a unitary flat vector bundle see: \cite{Fu,Fu2,CD1,CD2}. This result has been generalised in many directions; see, for example, \cite{CK} for the case $\dim B>1$, and \cite{I,LS} for the direct image of the relative pluricanonical bundles $f_*(\omega_{X/B}^{\otimes m})$.

Bo Berndtsson, in the seminal work \cite{bo}  shows that $f_*(\omega_{X/B}\otimes L)$ is semi-positive in the sense of Nakano if $L$ is semi-positive and this provides an evidence for  the Griffiths conjecture \cite{G}. Later in \cite{bo2}, it is shown that it is possible to relate the positivity of $f_*(\omega_{X/B}\otimes L)$ to the vanishing of the Kodaira-Spencer class. In particular, if $L=\sO_X$ these results highlight a close relation between curvature formulas and the classical infinitesimal Torelli problem. 

Motivated by these results, we study how to relate the complex analytic approach stemming from \cite{bo,bo2} to the theory of Massey products in the algebraic geometry context, cf. \cite{RZ4}. 

Massey products have been introduced in \cite{CP,PZ} and have been useful for solutions of Torelli-type problems; see: \cite{RZ1,RZ2,RZ3}. More importantly, they have been used in \cite{CP,RZ4} to study the unitary flat summand $\sU$ of $f_*\omega_{X/B}$ following an approach somewhat parallel to the metric one used by Berndtsson; see also \cite{BPW}.

For convenience we state our results in the case $\dim B=1$, nevertheless most of this work can be generalised to the case $\dim B>1$; see for example \cite{R} for the theory of Massey products in the case $\dim B>1$. 

The first result concerns a metric interpretation of the liftability of relative volume forms. To our best knowledge this is the first method to study the liftability of a Massey product that belongs to the ample part $\sA$ of the Fujita decomposition.

More precisely, let $\sW\subset f_*\omega_{X/B}$ be the standard subbundle generated via the Massey product technique, see Definition \ref{omegai}, and take on  $\sQ:= f_*\omega_{X/B}/\sW$ the quotient metric of the Hermitian metric defined in \cite{gt}. In Proposition \ref{formula} we give a formula computing the Chern curvature $\Theta_\sQ$  on the class $[\alpha]$, where $\alpha$  is a Massey product in $f_*\omega_{X/B}$.
From this explicit formula it becomes easy to show that the cup product of the Kodaira-Spencer class and a Massey product $\alpha$ in the ample part of the Fujita decomposition vanishes if and only if the curvature of $\sQ$ degenerates along $[\alpha]$.

\begin{thm}
	Let $\alpha$ be a Massey product in $\sA$. It is infinitesimally liftable 
	if and only if $\langle\Theta_\sQ [\alpha],[\alpha]\rangle_\sQ=0$.
%	 where $[\alpha]$ is its class in the quotient bundle $\sQ= f_*\omega_{X/B}/\sW$.
%	Let $\eta_{i}\in \Gamma(\Delta,\mD^1)$, then, for every $t\in \Delta$, $\alpha_t$ vanishes by cup product with the Kodaira-Spencer class  if and only if $\langle\Theta_\sQ [\alpha_t],[\alpha_t]\rangle_\sQ=0$. 
%	If we assume that $\eta_{i}\in \Gamma(\Delta,\mD^1\otimes \sO_B)$ and that their Massey product $\alpha$ is a holomorphic section of the ample part of the Fujita decomposition, then $\alpha_t$ vanishes by cup product with the Kodaira-Spencer class if and only if $\langle\Theta_\sQ [\alpha_t],[\alpha_t]\rangle_\sQ=0$.
\end{thm}
For more detailed statements see Corollaries \ref{curvagg} and \ref{curvagg2}.

One key property of this kind of metric results is that they can be meaningfully extended to the whole base $B$ using the notion of singular Hermitian metric, see \cite{P,raufi} and references therein for all the definitions.
Using a curvature formula from \cite{bo2}, we give conditions for the isotriviality of $f\colon X\to B$ over a Zariski open subset of $B$:
\begin{thm}\label{XYth}
Let $f\colon X\to B$ be a family of canonically polarised manifolds. Assume that there exists a surjective morphism $\rho\colon V\times B\to X$, where $V$ is a projective variety, such that $p:=f\circ \rho$ is the projection on $B$. Then $X$ is Zariski locally trivial.
\end{thm} 
See Theorem \ref{canpol}. This result must be seen in the light of the \lq constancy\rq\ results that are at the hearth of the recent solution, in some cases, of the Bombieri-Lang conjecture: we think that Theorem \ref{XYth} gives an interesting link between metric estimates and the techniques used in the proof of \cite[Subsection 2.3 and Theorem 2.6]{XY}. 

On the opposite side, our next result is a new method to study strong non isotriviality based on metric evaluations on the conjugate of the Kodaira-Spencer class. It arises from the study of the curvature of certain torsion free subsheaves of  $K^i_{\textnormal{prim}}\otimes \omega_B(E)^{\otimes i}$, where $K^i_{\text{prim}}$ is the sheaf of primitive classes in the kernel of $R^if_*\Omega^{n-i}_{X/B}(\log)\to \omega_B(E)\otimes R^{i+1}f_*\Omega^{n-i-1}_{X/B}(\log)$ and $E$ is the divisor on $B$ of the $f$-singular values.

Following \cite{Kov}, we give the following definition:

\begin{defn}
A semistable fibration $f\colon X\to B$ is said to be strongly non isotrivial if the morphism
$$
\rho_f\colon T_B(-E)^{\otimes n}\to R^nf_*\wedge^n T_{X/B}(-\log )
$$
given by the iteration of the Kodaira-Spencer map is not trivial.
\end{defn}

Actually we study the behaviour of the morphisms 
	$$\varphi_i\colon f_*\omega_{X/B}\otimes T_B(-E)^{\otimes i}\to R^if_*\Omega^{n-i}_{X/B}(\log)\quad i=1,\dots,n$$ 
	induced by iteration of the contraction with the Kodaira-Spencer class.
In particular if $\varphi_n$ is not zero then $f$ is strongly non isotrivial.
%
%Our next result is a new method to study strong non-isotriviality based on metric evaluations on the conjugate of the Kodaira-Spencer class and arises from the study of the curvature of certain torsion free subsheaves of  $K^i_{\textnormal{prim}}\otimes \omega_B(E)^{\otimes i}$, where $K^i_{\text{prim}}$ is the sheaf of primitive classes in the kernel of $R^if_*\Omega^{n-i}_{X/B}(\log)\to \omega_B(E)\otimes R^{i+1}f_*\Omega^{n-i-1}_{X/B}(\log)$.

We denote by $k_t$ the representative of the Kodaira-Spencer class $\xi_t$ coming from the so called horizontal lift of the local vector field $\frac{\partial}{\partial t}$ on $B$; see \cite{Siu,Sc,BPW} and Section \ref{sez1} for details. We denote a local section of $R^if_*\Omega^{n-i}_{X/B}(\log)$ by $[u]$; it can be seen as a function that maps a point $t\in B$ to a cohomology class $[u_t]$ of $(n-i,i)$-forms on the corresponding fiber $X_t$. We denote by $u_{t,h}\in [u_t]$ the harmonic representative of $[u_t]$ and by $c(\psi)$ the curvature  of a smooth metric on $\omega_B(E)$  of local weight $\psi$. We prove:

\begin{thm}\label{terzo}
If for every section  $[u]$ of $\varphi^i(\sA\otimes T_B(-E)^{\otimes i})\cap K^i_{\textnormal{prim}}$, $i=1,\dots,n-1$, the following inequality holds over $B\setminus E$ 
$$\lVert(\bar{k}_t\cup u_t)_h\lVert^2\geq ic(\psi)\lVert [u_t]\lVert^2$$
then either $f_*\omega_{X/B}$ is unitary flat or $f\colon X\to B$ is strongly non isotrivial.
\end{thm} See Theorem \ref{ineqk}.
Theorem \ref{terzo} must be put in correspondence to the semi-negativity results given in \cite{Z} and \cite{BPW}. 

Thanks to this new metric point of view we obtain a somewhat surprising result on surfaces $S$ of general type with positive index $\tau:=\frac{1}{3}(c_1^2(S)-2c_2(S))$, where $c_i(S)$, $i=1,2$, are the Chern classes of  $S$. 
Algebraic surfaces with $\tau>0$, even in the non simply connected case, have been studied; we can quote here \cite{In} for some examples and \cite{LY} for interesting geometrical properties.

\begin{thm}\label{indicepositivointro}
Let $f\colon X\to B$ be a semistable fibration and denote by $r$ the rank of $K^0_\textnormal{prim}$. Assume that the general fiber $X_t$ is a surface satisfying $K^2_{X_t}> 8\chi(\sO_{X_t})+2r+1$, then $f$ is strongly non isotrivial.
\end{thm} See Theorem \ref{indicepositivo}. This Theorem confirms that conditions for strong non isotriviality are actually related to upper bounds on the dimension of the vector subspace of liftable top forms, see Remark \ref{upperb}.
\begin{ackn}
	The first author has been supported by the Institute for Basic Science (IBS-R003-D1) and by European Union funds, NextGenerationEU. The second author has been supported by the grant DIMA Geometry PRIDZUCC, by PRIN 2017 Prot. 2017JTLHJR \lq\lq Geometric, algebraic and analytic methods in arithmetics\rq\rq and by the project DM737 RIC COLLAB ZUCCONI \lq\lq Birational geometry, Fano manifolds and torus actions\rq\rq CUP G25F21003390007, European Union funds, NextGenerationEU.
	
	The authors are members INdAM-GNSAGA.
\end{ackn}

\section{Setting}
\label{sez1}
Let $X$ be a smooth complex projective $n+1$-dimensional variety and $B$ a smooth complex projective curve.
In this paper we consider semistable fibrations $f\colon X\to B$, that is $f$ is a proper surjective morphism with connected fibers $X_t:=f^{-1}(t)$, $t\in B$, and such that the singular fibers are reduced and normal crossing. We remark that, thanks to the Semistable reduction theorem, see for example \cite{KKMSD}, we can always reduce to this case up to a sequence of blow-ups of $X$ and cyclic Galois coverings of $B$.

 We denote by $E\subset B$ the divisor of singular values of $f$ and by $S$ the (support of the) inverse image $f^{-1}(E)$.

In this Section, we give a characterisation of some of the key concepts of \cite{RZ4,RZ5,R} by means of their metric properties.

\subsection{Preliminary definitions}
We begin by recalling the short exact sequence 
\begin{equation}
\label{seq}
0\to f^*\omega_B\to \Omega^1_X\to \Omega^1_{X/B}\to 0
\end{equation} which defines the sheaf of holomorphic {relative differentials} $\Omega^1_{X/B}$. This sheaf in general is not locally free and for this reason it is often convenient to consider the logarithmic version of (\ref{seq}):
\begin{equation}
\label{seqlog}
0\to f^*\omega_B(E)\to \Omega^1_X(\log S)\to \Omega^1_{X/B}(\log)\to 0.
\end{equation}
We briefly recall that if $S$ is locally given by $z_1z_2\cdots z_k=0$ in appropriate local coordinates, the sheaf $\Omega^1_X(\log S)$ of {logarithmic differentials} is the locally free $\sO_X$-module generated by $dz_1/z_1,\ldots,dz_k/z_k,dz_{k+1},\ldots,dz_{n+1}$.

In the case of a semistable fibration, (\ref{seqlog}) is an exact sequence of vector bundles and 
%It is easy to see that we have the inclusion $\Omega^1_{X/B}\hookrightarrow \Omega^1_{X/B}(\log)$, so that  $\Omega^1_{X/B}$ is at least torsion free in our setting. \textcolor{red}{VEDERE}
% and its dual $T_X(-\text{log }D)$ is the subsheaf of the tangent sheaf $T_X$ consisting of those derivations that preserve the ideal sheaf $\sO_X(-D)$.
its $p$-wedge product, for $p=1,\dots, n$, is
\begin{equation}
\label{seqlogp}
0\to f^*\omega_B(E)\otimes \Omega^{p-1}_{X/B}(\log)\to \Omega^p_X(\log S)\to \Omega^p_{X/B}(\log)\to 0
\end{equation}  and  the determinant sheaf of $\Omega^1_{X/B}(\log)$ is 
$$
\det(\Omega^1_{X/B}(\log))=\det(\Omega^1_X(\log S))\otimes f^*\omega_B(E)^\vee=\omega_X(S)\otimes f^*\omega_B(E)^\vee=\omega_X\otimes f^*\omega_B^\vee=:\omega_{X/B}
$$ the {relative dualizing sheaf} of $f\colon X\to B$.

The pushforward via $f$ of Sequence (\ref{seqlogp}) is a long exact sequence of locally free sheaves on $B$, see \cite[Theorem 2.11]{S} cf. \cite[Lemma 2.11]{K},
%\begin{equation}
%0\to \omega_B(E)\otimes f_*\Omega^{p-1}_{X/B}(\log)\to f_*\Omega^p_X(\log W)\to f_*\Omega^p_{X/B}(\log)\to \omega_B(E)\otimes R^1f_*\Omega^{p-1}_{X/B}(\log)\to \dots
%\end{equation}
and we give the following definition. 
\begin{defn}
\label{defk}
We call $K^i$ the kernel of the connecting morphism $$R^if_*\Omega^{n-i}_{X/B}(\log)\to \omega_B(E)\otimes R^{i+1}f_*\Omega^{n-i-1}_{X/B}(\log).$$ %\textcolor{red}{Alternatively, $K^i$ is the image of $R^if_*\Omega^{n-i}_{X}(\log S)\to R^if_*\Omega^{n-i}_{X/B}(\log)$.}
\end{defn}
\begin{rmk}\label{remarkks}
At general point $t\in B$, the fiber of $K^i$ is the kernel of the homomorphism 
$$
H^i(X_t, \Omega^{n-i}_{X_t})\to T_{B,t}^\vee\otimes H^{i+1}(X_t, \Omega^{n-i-1}_{X_t})
$$
given by the cup product with the Kodaira-Spencer class $\xi_t\in H^1(X_t, T_{X_t})$.
\end{rmk}

%\begin{defn}\label{Ki}
%We call $\phi_i\colon f_*\omega_{X/B}\to \omega_B(E)^{\otimes i}\otimes R^if_*\Omega^{n-i}_{X/B}(\log)$ the composition\begin{equation*}
%f_*\omega_{X/B}\to  \omega_B(E)\otimes R^1f_*\Omega^{n-1}_{X/B}(\log)\to \omega_B(E)^{\otimes 2}\otimes R^2f_*\Omega^{n-2}_{X/B}(\log)\to \dots\to \omega_B(E)^{\otimes i}\otimes R^if_*\Omega^{n-i}_{X/B}(\log)
%\end{equation*} of the twisted connecting morphisms obtained from the direct images of (\ref{seqlogp}).
%\end{defn}

Following \cite{RZ4}, we also associate to $f\colon X\to B$ a local system of vector spaces on $B$, that is a sheaf on $B$ which is locally isomorphic to a constant sheaf. Denote by $\Omega^1_{X,d}$  the sheaf of de Rham closed holomorphic $1$-forms on $X$, hence $\Omega^1_{X,d}\subset\Omega^1_{X}(\log S)$ and we consider the composition
$$
f_*\Omega^1_{X,d}\hookrightarrow f_*\Omega^1_{X}(\log S)\to f_*\Omega^1_{X/B}(\log).
$$ The image sheaf of $f_*\Omega_{X,d}^{1}\to f_*\Omega^1_{X/B}(\log)$ is a local systems, see \cite[Lemma 2.6]{RZ4}.
\begin{defn}
	\label{defd}
	We call $\mD^1$ the local system obtained as image of the morphism $f_*\Omega_{X,d}^{1}\to f_*\Omega^1_{X/B}(\log)$.
\end{defn} 
Equivalently, $\mD^1$ fits into the following short exact sequence 
\begin{equation}
	\label{dx}
	0\to \omega_B\to f_*\Omega^1_{X,d}\to\mD^1\to 0.
\end{equation}

Similarly, we denote by $\mD^{n}$ the local system contained in $f_*\Omega^{n}_{X/B}(\log)=f_*\omega_{X/B}$ and defined by the de Rham closed $n$-forms. As shown in \cite[Theorem 3.7]{RZ4}, $\mD^{n}$ is actually the local system of the so called second Fujita decomposition of $f_*\omega_{X/B}$. In fact, as recalled in the Introduction, there is a decomposition 
\begin{equation}
\label{fujita}
f_*\omega_{X/B}=\sU\oplus\sA
\end{equation} where $\sA$ is an ample vector bundle (or $\sA=0$) and $\sU$ is a unitary 
flat vector bundle; $\mD^{n}$ is such that $\sU=\mD^{n}\otimes_\mC \sO_B$. 

Finally, we point out that we have the inclusions $\mD^n\subset \sU\subseteq K^0\subseteq f_*\omega_{X/B}$.
% The first inclusion is obvious by what stated above, for the second just notice that $K^0=\Ima (f_*\Omega^{n}_{X}(\log W)\to f_*\Omega^{n}_{X/B}(\log))$ and $\mD^n=\Ima (f_*\Omega^{n}_{X,d}\to f_*\Omega^{n}_{X/B}(\log))$. In particular $\mD^n\subset K^0$ and tensoring by $\sO_X$ we get the desired inclusion. 
The difference between the rank of $\sU$ and the rank of $K^0$ is  of interest, see for example \cite{GT}.

\subsection{Hermitian metrics on direct images} 
%We introduce the Hermitian vector bundles that we need in the following.
%To do this, let us c
Consider  the open subset  $B\setminus E$ corresponding to the smooth fibers. The vector bundle   $f_*\omega_{X/B}$ and, more in general, the bundles of primitive cohomology classes  $\sP^{n-i,i}\subset R^if_*\Omega^{n-i}_{X/B}(\log)$ are endowed with a natural smooth Hermitian metric coming from the Hodge metric on the fibers, cf. \cite{G2,gt}. Actually, one can also consider the twisted cases $f_*(\omega_{X/B}\otimes L)$ and $R^if_*(\Omega^{n-i}_{X/B}(\log)\otimes L)$ where $L$ is a Hermitian line bundle on $X$ whose curvature satisfies suitable properties. See \cite{bo,bo2} for the case where $L$ is semi-positive and \cite{BPW} for the case where $L$ is semi-negative.

From now on we fix a K\"ahler form $\Omega$ on $X$. 
We denote by $\omega_t$ the restriction of $\Omega$ to $X_t$.

Following \cite{bo,bo2}, we first briefly recall the definition of the metric on $f_*\omega_{X/B}$ and some of its key properties. 
Consider $u$ a local section of $f_*\omega_{X/B}$ over a disk $\Delta\subset B\setminus E$. It can be seen as a function that maps a general point $t\in \Delta$ to a global holomorphic $(n,0)$-form $u_t$ on the corresponding fiber $X_t$. 
The Hermitian norm is then defined as
\begin{equation}
	\label{norma1}
	\lVert u_t\lVert^2_t=c_n\int_{X_t}  u_t\wedge\bar{u}_t
\end{equation} where $c_n=i^{{n}^2}$ is a unimodular constant.

Key computations on this metric strongly rely on the notion of a good {representative} of a section. 
A {representative} of $u$, denoted by $\bu$, is a smooth $(n,0)$-form on $f^{-1}(\Delta)\subset X$ which restricts to $u_t$ on the fibers.

%On $B\setminus E$, the bundle $f_*\omega_{X/B}$ is therefore a Hermitian vector bundle and the key point is that we have an explicit expression for its Chern connection $D = D'+ D''$ and its curvature $\Theta$ in terms of the representative $\bu$. 
The key point is that we have an explicit expression for the Chern connection $D = D'+ D''$ and  curvature $\Theta$ in terms of the representative $\bu$.
Indeed, we can write 
$$
\partial\bu=\mu\wedge dt,
$$
and, since $\bar{\partial}\bu$ is trivial along the fibers, we can write 
$$
\bar{\partial}\bu=\nu\wedge d\bar{t}+\eta\wedge dt,
$$ 
where $\nu$ and $\mu$ are relative forms of bidegree $(n,0)$ and $\eta$ is of bidegree $(n-1,1)$.
 We have that 
\begin{equation}\label{connanti}
(D''u)_t=\nu_t d\bar{t}
\end{equation} and 
\begin{equation}\label{connolo}
(D'u)_t=P(\mu_t)dt
\end{equation} where $P$ is the orthogonal projection of $(n, 0)$-forms on the space of holomorphic $(n, 0)$-forms.
The $(n-1,1)$-form $\eta_t$ has a neat interpretation in terms of the cup product with the Kodaira-Spencer class $\xi_t$: it is a representative of the cohomology class of the cup product $\xi_t\cup u_t\in H^{n-1,1}(X_t)$. See \cite{bo} for all the details on these formulas.

Of course the choice of the representative $\bu$ is not unique; two such choices differ by a term of the form $v\wedge dt$, but everything mentioned so far is independent from this choice.

Nevertheless if $u$ is a holomorphic section of $f_*\omega_{X/B}$ and $t$ is a point in $\Delta$, by \cite{bo,bo2}, there is a preferred choice of $\bu$ such that $\mu_t$ is holomorphic (i.e. $P(\mu_t)=\mu_t$) and $\eta_t$ is primitive. The key advantage of  this choice is that we get $(D'u)_t=\mu_t dt$  and  $-c_n\int_{X_t}\eta_t\wedge\bar{\eta}_t=\lVert\eta_t\lVert^2$. This choice of representative allows to compute the Chern curvature formula for the Hermitian metric as follows:
\begin{equation}
\label{curvatura1}
\langle\Theta u_t,u_t \rangle_t=\lVert\eta_t\lVert^2.
\end{equation}  See \cite{gt} for the original proof.  Berndtsson in \cite[Section 3]{bo2} proves that $\eta_t$ is the unique harmonic representative of the cohomology class $\xi_t\cup u_t\in H^{n-1,1}(X_t)$, hence we can also write this formula as
\begin{equation}
\label{curvatura2}
\langle\Theta u_t,u_t \rangle_t=\lVert\xi_t\cup u_t\lVert^2.
\end{equation} 
%\textcolor{red}{In particular $f_*\omega_{X/B}$ is Griffiths semi-positive.}

In the case of $\sP^{n-i,i}$, the situation is similar, but we have to be a little more careful since a section $[u]$ is a function that maps $t$ to a primitive cohomology class of $(n-i,i)$-forms on the corresponding fiber $X_t$. For this part we follow \cite{BPW}.

We denote by $u_{t,h}$ the harmonic representative of this class. By a fundamental result of Kodaira-Spencer, the variation of $u_{t,h}$ with respect to $t$ is
smooth.
The norm is
\begin{equation}
\lVert[u_t]\lVert^2_t=(-1)^ic_n\int_{X_t} u_{t,h}\wedge \bar{u}_{t,h}.
\end{equation}
As in the $(n,0)$ case, we consider a representative $\bu$ of $[u]$, that is a
smooth $(n-i,i)$-form on $f^{-1}(\Delta)$ which is $\bar{\partial}$-closed on the fibers of $f$
and whose restriction to each $X_t$ belongs to the cohomology class $[u_t]$. In this case the preferred choice of such a representative is given by the so called {vertical representative}.
\begin{defn}\label{vert}
We say that a smooth vector field $V$ on $X$ is {horizontal} if $df(V)$ is a non-zero $(1,0)$ vector field and $\Omega(V,\bar{U})=0$ for every other vector field $U$ with $df(U)=0$. 
We say that a form $\bu$ is {vertical} if the contraction $V|\bu$ is equal to 0 for any horizontal vector field $V$. 
\end{defn}
The vertical representative is unique, see \cite[Proposition 5.1]{BPW}; we consider $\bu$ the vertical representative of $[u]$.
 Similarly as before, we write
$$
\bar{\partial}\bu=\nu\wedge d\bar{t}+\eta\wedge dt
$$ and
$$
\partial\bu=\zeta \wedge d\bar{t}+\mu\wedge dt.
$$ 
Then $\nu$ and $\mu$ are relative forms of bidegree $(n-i,i)$, $\eta$ is of bidegree $(n-i-1,i+1)$ and $\zeta$ is of bidegree $(n-i+1,i-1)$; they are primitive on the fibers, see \cite[Section 7]{BPW}.
As before we can explicit compute the Chern connection as
\begin{equation}
(D''[u])_t=[\nu_t] d\bar{t}
\end{equation} and
\begin{equation}
(D'[u])_t=[\mu_t] dt.
\end{equation}

As a representative of the Kodaira-Spencer class $\xi_t$ we choose the (0,1)-form with coefficients in the tangent space $T_X$ given by $\bar{\partial}V$ where $V$ is the horizontal lift of the $(1,0)$ vector $\partial/\partial t$ on $\Delta$. That is, on $X_t$  we denote $k_t:=\bar{\partial}V_{|X_t}$ and $\xi_t=[k_t]\in H^1(X_t,T_{X_t})$.
It turns out that with this choice 
\begin{equation}
\eta_t=k_t\cup u_t
\end{equation}
and
\begin{equation}\label{zeta}
\zeta_t=\bar{k}_t\cup u_t
\end{equation}
if $u_t$ is harmonic.
\begin{rmk}
	In terms of cohomology classes, this means that
	$$
	[\eta_t]=\xi_t\cup [u_t]
	$$  and 
	$$
	[\zeta_t]=\overline{(\xi_t\cup\overline{[u_t]})}
	$$ since the contraction with $\bar{k}_t$ corresponds to the composition
	$$
	H^{p,q}\cong H^{q,p}\xrightarrow{\xi_t\cup\  } H^{q-1,p+1}\cong H^{p+1,q-1}
	$$ where the isomorphisms are given by complex conjugation.
\end{rmk}

The curvature formula generalising (\ref{curvatura1}) is
\begin{equation}
\label{curvatura3}
\langle\Theta [u_t],[u_t] \rangle_t=\lVert\eta_{t,h}\lVert^2-\lVert\zeta_{t,h}\lVert^2
\end{equation} where by $\eta_{t,h}$ and $\zeta_{t,h}$ we mean the harmonic part of $\eta_t$ and $\zeta_t$ respectively, see \cite{BPW}.

One of the immediate consequences of the  curvature formulas (\ref{curvatura1}) and (\ref{curvatura3}) is that the bundles $K^i_{\text{prim}}:=K^i\cap \sP^{n-i,i}$ are semi-negatively curved since by definition they are the primitive part of the kernel of the cup product with the Kodaira-Spencer class. For example if $u$ is a section of $K^0_{\text{prim}}=K^0$, then by (\ref{curvatura1}) we have that $\langle\Theta u_t,u_t \rangle_t=0$ and for $[u]$ a section of $K^i_{\text{prim}}$, $i\geq 1$, by (\ref{curvatura3}) we have that $\langle\Theta [u_t],[u_t] \rangle_t=-\lVert\zeta_{t,h}\lVert^2\leq0$. In both cases, since curvature decreases in subbundles, we have that $K^i_{\text{prim}}$ is semi-negatively curved on $B\setminus E$ for $i\geq0$. 

%Actually by \cite[Section 9]{BPW}, there is an extension across the singularities, that is the smooth Hermitian metric on the open part of $B$ just described extends on the whole curve $B$ to a {singular Hermitian metric} which is  semi-negatively curved in the sense of singular metrics. 
{Actually there is an extension across the singularities. Indeed $\sP^{n-i,i}$ and $K^i_{\text{prim}}$ are firstly defined on $B\setminus E$, but since they are subsheaves of $R^if_*\Omega^{n-i}_{X/B}(\log)$, which is defined over the whole base $B$, they have a natural extension across $E$. Furthermore by \cite[Section 9]{BPW} the chosen smooth Hermitian metric on $B\setminus E$ extends to a {singular Hermitian metric}. This singular metric is semi-negatively curved on $K^i_{\text{prim}}$ in the sense of singular metrics.}
%, see for example the Introduction or \cite[Definition 1.2]{raufi} for the definition of semi-negatively curved singular metric.

\subsection{Metric interpretation}
The Chern connection $D$ on the vector bundle $f_*\omega_{X/B}$ is compatible with the Gauss-Manin connection. This is shown for example in \cite[Section 2.3]{bo2}. Here we rephrase it using the objects introduced in the beginning of this section. Using the same notation as above, we denote by $\Theta$ the Chern curvature of $f_*\omega_{X/B}$.
\begin{prop}
\label{caratt}
Consider the vector bundle $f_*\omega_{X/B}$ with the Hermitian structure described above. The subsheaves $\mD^n\subset \sU\subseteq  K^0$ of $f_*\omega_{X/B}$ are characterised as follows.
\begin{enumerate}[label=(\roman*)]
\item $\mD^n$ is the kernel of the Chern connection $D$
\item $\sU$ is the largest flat subbundle of $f_*\omega_{X/B}$
\item $K^0$ is the kernel of the curvature  $\Theta$.
\end{enumerate}
In particular $K^0$ is a (holomorphic) direct summand of $f_*\omega_{X/B}$ if and only if $K^0=\sU$.
\end{prop} 
\begin{proof}
\textit{(i) }First note that  $\mD^n=\Ima (f_*\Omega^n_{X,d}\to f_*\omega_{X/B})$. In particular if $u$ is a local section of $\mD^n$ it follows that, at least locally on the base $B$, we can choose the representative $\bu$ to be a holomorphic de Rham closed $n$-form. Hence $\partial \bu=\bar{\partial}\bu=0$ and we get that $\mD^n$ is  contained in the kernel of the connection $D$ by (\ref{connanti}) and (\ref{connolo}). Viceversa if $u$ is a holomorphic section in the kernel of $D$, by (\ref{curvatura1}) we have that $\eta_t=0$ for all $t$, so that $\xi_t\cup u_t=0$. This means that $u$ is in $K^0$ by Remark \ref{remarkks}, hence, since $K^0$ can also be seen as the image of $f_*\Omega^{n}_{X}\to f_* \omega_{X/B}$, we can choose $\bu$ to be a section of $f_*\Omega^n_{X}$, so that $\bar{\partial}\bu=0$.  Finally, since $D'u=0$ we also have that $\mu=0$, that is $\partial \bu=0$.
So we conclude that $\bu$ is de Rham closed and $u$ is a local section of $\mD^n$.
%we take its preferred representative $\bu$ as described in the previous subsection. Then, by Equations (\ref{connanti}) and (\ref{connolo}), we have that $\nu=0$ since $D''u=0$, $\mu=0$ since $D'u=0$ and, by (\ref{curvatura1}), $\eta=0$. 

\textit{(ii) }By point \textit{(i)}, this follows easily since $\sU=\mD^n\otimes \sO_B$.
Indeed note that $\sU$ is contained both in the kernel of the curvature $\Theta$ and in the kernel of the second fundamental form of $\sU$ in $f_*\omega_{X/B}$.
%
% that is $\Theta_{|\sU}=0.$ Furthermore we know that $\sU$ is a factor in a direct sum decomposition of $f_*\omega_{X/B}$, see Equation (\ref{fujita}). Hence the second fundamental form vanishes and we have that the curvature of $\sU$ (as a subbundle of $f_*\omega_{X/B}$) vanishes identically. 

\textit{(iii) } If $u$ is a holomorphic section of $K^0$, locally on $B$ the representative $\bu$ can be chosen as a holomorphic form, as in point \textit{(i)}. So $\bar{\partial}\bu=0$ and $\partial \bu=\mu\wedge dt$, where $\mu$ is a holomorphic section of $f_* \omega_{X/B}$. Hence $\Theta u=D''D'u=0$. Viceversa if $u$ is a holomorphic section of $f_*\omega_{X/B}$ with $\Theta u=0$,  by (\ref{curvatura2}) $u_t$ is in the kernel of the cup product with $\xi_t$ for all $t$, hence $u$ is a section of $K^0$, by Remark \ref{remarkks}.
 \end{proof}

 \begin{rmk}
We have already seen that $K^0$ is semi-negatively curved with the metric induced from $f_*\omega_{X/B}$. It seems actually interesting to understand when $K^0$ is {maximal} inside $f_*\omega_{X/B}$ with this property. In Section \ref{sezkeriterati} for example we study the kernels of the iterated cup product with the Kodaira-Spencer class.
\end{rmk}

 \section{Massey product and curvature formula}
\label{sez2}
Using curvature formulas we can give conditions on liftability of Massey products.
First we briefly recall this notion, which is one of the main constructions of \cite{PT,RZ4}. If a Massey product $\alpha_t$ is Massey trivial, see Definition \ref{mastrivpoint} below, then its cup product with the Kodaira-Spencer class $\xi_t$ is zero, that is $\alpha_t$ is liftable. The consequences of this are well studied in relation to Torelli-type problems.  
On the other hand  the case where $\alpha_t$ is not Massey trivial, but still liftable is not well understood, especially when the Massey product is a section of the ample part $\sA$ of the Fujita decomposition (\ref{fujita}). Hence below we give a  metric interpretation of this condition. To our best knowledge it is the first time that such an interpretation is studied.

 From the direct image of Sequence (\ref{seq}), we get a short exact sequence on $B$ defining the vector bundle $K_\partial$:
\begin{equation}
	\label{kdelta}
	0\to \omega_B\to f_*\Omega^1_X\to K_\partial\to 0.
\end{equation} 
Intuitively we can think of $K_\partial\subseteq f_*\Omega^1_{X/B}$ as the sheaf of holomorphic one forms on the fibers of $f$ which are liftable to a tubular neighbourhood in $X$ of such fibers. 
%If $n=1$, this coincides (at least on an open set) with $K^0$. 
For a more in depth study of $K_\partial$ see \cite{PT,GST,GT} for the case $n=2$, \cite{RZ4} for the general case. 

Note that if in Sequence (\ref{kdelta}) we consider, instead of $f_*\Omega^1_X$, the sheaf $f_*\Omega^1_{X,d}$, we obtain the exact sequence (\ref{dx})
%\begin{equation}
%	\label{seqd1}
%	0\to \omega_B\to f_*\Omega^1_{X,d}\to \mD^1\to 0
%\end{equation}  
and recover the local system $\mD^1$ as introduced in Section \ref{sez1}. 
%Actually by (\ref{seqd1}), $\mD^1$ is a subsheaf of $f_*\Omega_{X/B}^1$, but since $f_*\Omega_{X/B}^1\hookrightarrow f_*\Omega_{X/B}^1(\log)$, this is completely equivalent to Definition \ref{defd}

By \cite[Lemma 3.5]{PT} or \cite[Lemma 2.2]{RZ4}, Sequence (\ref{kdelta}) splits, and from now on we choose and fix one such splitting.
Now consider  $n+1$ linearly independent sections $\eta_1,\ldots,\eta_{n+1}\in \Gamma(\Delta,K_{\partial})$ on an open subset $\Delta\subset B$ of coordinate $t$ and denote by $s_1,\dots,s_{n+1}\in \Gamma(\Delta,f_*\Omega^1_{X})$ the liftings of $\eta_1,\dots,\eta_{n+1}$ according to the chosen splitting of Sequence (\ref{kdelta}).

\begin{defn}
	\label{omegai}
	We call $\alpha_i\in \Gamma(\Delta, f_*\omega_{X/B})$, $i=1,\ldots,{n+1}$, the relative forms defined by $\eta_1\wedge\dots\wedge\hat{\eta_i}\wedge\dots\wedge\eta_{n+1}$ and $\sW$ the (local) $\sO_B$-submodule of $f_*\omega_{X/B}$ generated by the $\alpha_i$.
\end{defn} Note that by $\hat{\eta_i}$ we mean that $\eta_i$ is excluded from the wedge product.
\begin{defn}
	\label{mtrivial}
	The {Massey product} of $\eta_1,\ldots,\eta_{n+1}$ is the section $\alpha\in \Gamma(\Delta, f_*\omega_{X/B})$ defined by $s_1\wedge...\wedge s_{n+1}=\alpha\wedge dt$. We say that the sections $\eta_1,\ldots,\eta_{n+1}$  are {Massey trivial} if $\alpha$ is  a section of the submodule $\sW$. 
\end{defn} 

\begin{rmk}
	\label{remlif}
%	Being Massey trivial means that locally in $f_*\omega_{X/B}$  we can write
%	$$
%	\alpha=\sum_i a_i\alpha_i
%	$$ for some $a_i\in\sO_B(\Delta)$. 
%	Alternatively,  if we look at things in $f_*\omega_X$, we can write
%	$$ 
%	s_1\wedge\dots\wedge s_{n+1}=\sum_i \alpha_i\wedge a_idt.
%	$$
%	As a section of $f_*\omega_{X/B}$, the Massey product $\alpha$ certainly depends on the choice of the splitting mentioned above. On the other hand, the condition of being Massey trivial does not; see \cite{RZ4}.
	
	 The crucial result is that if $\eta_1,\ldots,\eta_{n+1}$  are Massey trivial, we can find 
	liftings $\tilde{s}_i$ such that 
	\begin{equation}\label{giusti}
		\tilde{s}_1\wedge\dots\wedge\tilde{s}_{n+1}=0;
	\end{equation} in particular we can take $\alpha$ to be zero. See for example \cite[Proposition 4.10]{RZ4}.
\end{rmk}

Since $\mD^1$ is a subsheaf of $K_{\partial}$, it makes sense to construct Massey products starting from sections of $\mD^1$, i.e.  $\eta_i\in \Gamma(\Delta,\mD^1)$. One of the key points in \cite{PT} and \cite{RZ4} is exactly to consider this setting.

%\textcolor{red}{
%The notion of Massey triviality can also be extended naturally to a vector space $U\subset \Gamma(\Delta, K_{\partial})$ of dimension at least ${n+1}$.
%\begin{defn}
%	\label{mastriv}
%	We say that $U\subset \Gamma(\Delta, K_{\partial})$ is Massey trivial if any ${n+1}$-uple of linearly independent sections in $U$ is Massey trivial (according to  Definition \ref{mtrivial}).
%\end{defn}
%We also recall the notion of strictness that we will need in the following.
%\begin{defn}
%	\label{strict}
%	We say that $U\subset \Gamma(\Delta, K_{\partial})$ is  strict if the map 
%	$$
%	\bigwedge^{n}U\otimes \sO_\Delta\to f_*\omega_{X/B|\Delta}
%	$$ is an injection of vector bundles.
%\end{defn}
%\begin{rmk}
%%	Given an $n$-dimensional variety $Y$, a subspace $U$ of $H^0(Y,\Omega^1_Y)$ is usually called strict if  $\bigwedge^n U$ injects into $H^0(Y,\omega_Y)$. 
%This definition is a relative version of \cite[Definition 2.1 and 2.2]{Ca2}.
%\end{rmk}}
%\textcolor{blue}{Se togliamo qui, togliere anche [C] dalla biblio}

The construction of Massey products can also be done pointwise, that is for a fixed regular value $t\in B$. This is the original approach of \cite{CP,PZ,RZ1}. 
In this case one considers  $\eta_{1,t},\ldots,\eta_{n+1,t}\in K_{\partial,t}\leq H^0(X_t,\Omega^1_{X_t})$ the restriction of the $\eta_i$ to the fiber $X_t$. Similarly $\sW_t=\langle\alpha_{1,t},\dots,\alpha_{n+1,t}\rangle\leq H^0(X_t,\omega_{X_t})$ and $\alpha_t\in H^0(X_t,\omega_{X_t})$ are restriction on $X_t$ of $\sW$ and of $\alpha$ respectively. We have a pointwise version of Definition \ref{mtrivial}
\begin{defn}\label{mastrivpoint}
	We say that $\eta_{1,t},\ldots,\eta_{n+1,t}$ are Massey trivial (at $t$) if $\alpha_t\in \sW_t$. 
\end{defn}

\subsection{A curvature formula on Massey products}
In this section we consider a disk $\Delta\subset B\setminus E$. Using the same notation as above, take sections  $\eta_1,\dots,\eta_{n+1}\in \Gamma(\Delta,K_{\partial})$ and $\alpha$ their Massey product. The notion of Massey triviality, see Definition \ref{mtrivial}, means that $\alpha$ is a section of the $\sO_B$-submodule $\sW=\langle\alpha_{1},\dots,\alpha_{n+1}\rangle\subseteq {f_*\omega_{X/B}}_{|\Delta}$, hence it is  natural to study the quotient $\sQ:={f_*\omega_{X/B}}_{|\Delta}/\sW$. Indeed  the class $[\alpha]$ does not depend on the choice of liftings $s_i$ of $\eta_i$, while the section $\alpha$ does. Following this idea, we assume that $\sW$ is a non-trivial subbundle of ${f_*\omega_{X/B}}_{|\Delta}$ and we compute the curvature of $\sQ$ on the section  $[\alpha]$
%More in details, denote by $\langle,\rangle_\sQ$ the quotient Hermitian metric on $\sQ$ and by $\Theta_\sQ$ its curvature.
%We compute the curvature formula $\langle\Theta_\sQ [\alpha_t],[\alpha_t]\rangle_\sQ$ 
and  relate it to the cup product with the Kodaira-Spencer class and the pointwise liftability of $\alpha$.
%$\xi_t\cup\alpha_t$. 
%This problem is of interest and referred to as (first order\textcolor{red}{/pointwise? pensare al caso in cui alpha non sta in W ma si annulla in un punto t}) {liftability} \textcolor{red}{(at t)} of  Massey products. 
%i.e. the existence of a lifting of $\alpha$ in a small tubular neighborhood of the fibers, and it has not been fully analyzed in
%\cite{RZ4}. 
The result is a version of Formula \ref{curvatura2}, the difference is that working on the quotient $\sQ$ makes the key role of the local system $\mD^1$ more clear.

In this Section, let $\sE:={f_*\omega_{X/B}}_{|\Delta}$  and, to avoid confusion, here we denote by $\langle,\rangle_\sE$ its metric introduced in Section \ref{sez1} and by $\langle,\rangle_\sQ$ the quotient metric on $\sQ$. 
%We have the exact sequence of holomorphic vector bundles on $\Delta$
%$$
%0\to \sW\to\sE\to\sQ\to0.
%$$
%We work for simplicity in the case $n=1$, that is $f\colon X\to B$ is a fibered surface. Of course the case $n\geq 2$ is very similar.

%\textcolor{red}{Consider $\eta_{1,t},\dots,\eta_{n+1,t}\in H^0(X_t,\Omega^1_{X_t})$ the restriction of the $\eta_i$ to $X_t$.} 
%\textcolor{red}{We assume for simplicity that  $U=\langle\eta_1,\dots,\eta_{n+1}\rangle$ is strict so that in particular $\alpha_{1,t},\dots,\alpha_{n+1,t}$ are linearly independent in $H^0(X_t,\omega_{X_t})$; this assumption is not actually neccessary.} 
{ By  the above discussion}, the case to be studied is when $\alpha_t$ is not contained in the vector space $\sW_t=\langle\alpha_{1,t},\dots,\alpha_{n+1,t}\rangle$.
% $\alpha_t$ is obviously liftable since all the $\alpha_{i,t}$ are.
Without loss of generality, we assume that, by a standard orthogonalisation process, $\alpha_{1,t},\dots,\alpha_{m,t},\alpha_t$, $m\leq n+1$, are a unitary basis of $\langle\sW_t,\alpha_t\rangle$ (note that  $\alpha_{1,t},\dots,\alpha_{n+1,t}$ are not necessarily linearly independent).
%; this does not change the subspace $\sW_t$ and it does not affect the liftability of $\alpha_t$.
%and we  identify $\alpha_t$ with its class in the quotient.
% and we call $\hat{\omega}^t$ the normalized projection of $\omega^t$ on the orthogonal of $\sW_b$, so that we identify $\hat{\omega}^t$ with its class in the quotient. 
We can extend to a $\sC^\infty$ unitary frame of $\sE$, denoted by $\alpha'_1,\dots,\alpha'_{m},\alpha',\tau_1,\dots,\tau_k$, such that $\alpha'_1,\dots,\alpha'_{m}$ is a smooth frame for $\sW$, $\alpha',\tau_1,\dots,\tau_k$ is a smooth frame for $\sW^\perp\cong\sQ$ and $\alpha'_1,\dots,\alpha'_{m},\alpha'$ restrict to $\alpha_{1,t},\dots,\alpha_{m,t},\alpha_t$ on $X_t$.
%this orthonormal basis.
%We complete $\alpha_{1,t},\dots,\alpha_{n+1,t},\alpha_t$ to an orthonormal basis of $H^0(X_t,\omega_{X_t})$ and
%$\eta_1,\eta_2$ local holomorphic sections of $f_*\omega_{X/B}$ and $\omega$ their Massey product.  We call  The module $\sW$ has a local holomorphic frame given by $\eta_1,\eta_2$. We also assume that $[\omega]$ is different from zero in $\sQ$, otherwise the our problem is trivial, so that a local holomorphic frame of $\sE$ can be obtained by completing $\eta_1,\eta_2,\omega$.
%We actually get from this an orthonormal smooth frame of $\sE$, which we denote by $\hat{\eta_1},\hat{\eta_2},\hat{\omega},\tau_1,\dots,\tau_k$.

We compute $\langle\Theta_\sQ [\alpha_t],[\alpha_t]\rangle_\sQ$ by means of the second fundamental form $\mathfrak{S}$ of $\sQ$ in $\sE$. Following the notation of \cite{GH}, the second fundamental form $\mathfrak{S}\colon \sA^0(\sQ)\to \sA^1(\sW)$ is defined as the composition of the connection $D$ on $\sE$ restricted to $\sQ\cong\sW^\perp$
$$
D_{|\sQ}\colon \sA^0(\sQ)\to \sA^1(\sE)
$$ and the projection to $\sW$
$$
\sA^1(\sE)	\to \sA^1(\sW).
$$
%Note that we have $[\omega^t]=\lambda[\hat{\omega}^t]$, $\lambda\in \mC$ given by normalization. Hence we will compute $\langle\Theta_\sQ [\hat{\omega}^t],[\hat{\omega}^t]\rangle_\sQ$.

We have 
$$
\langle\Theta_\sQ [\alpha_t],[\alpha_t]\rangle_\sQ=\langle\Theta_\sE \alpha_t,\alpha_t\rangle_\sE-\langle \mathfrak{S} \alpha_t,\mathfrak{S}\alpha_t\rangle_\sE
$$
 The first of these summands is given by (\ref{curvatura2}), that is 
 $$
 \langle\Theta_\sE \alpha_t,\alpha_t\rangle_\sE=\lVert\xi_t\cup {\alpha}_t\lVert^2
 $$ hence we are left with the second summand.

In our chosen unitary frame we write $\mathfrak{S}\alpha_t=\lambda_1\alpha_{1,t}+\dots+\lambda_{m}\alpha_{m,t}$, for certain $(0,1)$-forms $\lambda_i$, \cite[Proposition 1.6.6]{Kob}. Explicitly, since $\sW$ is a holomorphic subbundle of $\sE$,
$$
\lambda_i=\langle(D\alpha')_t,\alpha_{i,t}\rangle_\sE=-\langle\alpha_t,(D\alpha'_{i})_t\rangle_\sE=-\left(\int_{X_b}c_n\alpha_t\wedge \overline{P(\mu'_i)}\right)d\bar t
$$ where we use the same notation as Section \ref{sez1}, Equation (\ref{connolo}).
\begin{prop}
\label{formula}
The following formula computes the curvature of $\sQ$ on the class $[\alpha_t]$
\begin{equation}
\langle\Theta_\sQ [\alpha_t],[\alpha_t]\rangle_\sQ=\lVert\xi_t\cup {\alpha}_t\lVert^2+\sum_i\Big|\int_{X_b}c_n\alpha_t\wedge \overline{P(\mu'_i)}\Big|^2.
\end{equation}
\end{prop}

As pointed out above, the case where the sections $\eta_{i}$ are in the local system $\mD^1$ is of particular interest.
The following corollary indeed specifies what happens in this case.
\begin{cor}\label{curvagg}
Let $\eta_{i}\in \Gamma(\Delta,\mD^1)$, then, for every $t\in \Delta$, $\alpha_t$ is liftable iff $\langle\Theta_\sQ [\alpha_t],[\alpha_t]\rangle_\sQ=0$. 
\end{cor}
\begin{proof}
Note that if the $\eta_{i}$ are in the local system $\mD^1$, then the $\alpha_{i}$ are in $\mD^n$, hence we can assume that the $\alpha_i'$ are holomorphic and by Proposition \ref{caratt} the summation in the formula of Proposition \ref{formula} is zero.
\end{proof}
\begin{rmk}
	In particular $\alpha$ is a section of $K^0$ if and only if $\langle\Theta_\sQ [\alpha_t],[\alpha_t]\rangle_\sQ=0$ for every $t\in \Delta$. Of course by Formula \ref{curvatura2}, $\alpha$ is a section of $K^0$ if and only if $\langle\Theta_\sE \alpha_t,\alpha_t\rangle_\sE=0$ for every $t\in \Delta$. The condition on the curvature of $\sQ$ is in general stronger in the sense that it points to  directions where the curvature of the quotient bundle is degenerate.
%	 but the condition on the curvature of $\sQ$ is in general stronger and has the advantage that $[\alpha_t]$ does not depend on the choice of liftings.
%	
%	Is a major indication of non ampleness of Q
%	direzioni di degen di curv di Q
%	
%	U=W+W' decomp ort
\end{rmk}
%Another interesting point not fully analyzed in
%\cite{RZ4}, \cite{R} is when the Massey product is a holomorphic section of the ample part $\sA$ of the Fujita decomposition. 

Another interesting point not fully analysed in
the literature is when the Massey product is a section of $\sA$, the ample part of the Fujita decomposition. In this case we can relax the assumption on the $\eta_{i}$ and obtain the same conclusion.
\begin{cor}\label{curvagg2}
	Assume that $\eta_{i}\in \Gamma(\Delta,\mD^1\otimes \sO_B)$ and that their Massey product $\alpha$ is a holomorphic section of $\sA$, then then $\alpha_t$ is liftable iff $\langle\Theta_\sQ [\alpha_t],[\alpha_t]\rangle_\sQ=0$.
\end{cor}
\begin{proof}
In this case the $\alpha_{i}$ are in $\mD^n\otimes \sO_B=\sU$ the unitary flat part of the Fujita decomposition. Since the decomposition $f_*\omega_{X/B}=\sU\oplus\sA$ is orthogonal, we get that $D\alpha'_i$ is not necessarily zero, but is orthogonal to $\alpha$. Hence once again the summation in the formula of Proposition \ref{formula} is zero.
\end{proof}

The above Corollaries give a metric interpretation of the liftability of Massey products. Based on our experience concerning the case of trivial Massey products, we expect these results to be useful in the future. 

\section{Isotriviality and strong non isotriviality}
\label{sezkeriterati}
In this section we use the metrics introduced in Section \ref{sez1} to study possible conditions for the triviality or, on the opposite side, strong non isotriviality of the fibration $f\colon X\to B$.

\subsection{Strongly non isotrivial fibrations}
We give the following definition
\begin{defn}\label{Ki}
	We call $\phi_i\colon f_*\omega_{X/B}\to  \omega_B(E)^{\otimes i}\otimes R^if_*\Omega^{n-i}_{X/B}(\log)$ the composition of the connecting morphisms obtained from the direct images of (\ref{seqlogp}):
	\begin{equation*} 
		\label{composizione}
	\phi_i\colon	f_*\omega_{X/B}\to  \omega_B(E)\otimes R^1f_*\Omega^{n-1}_{X/B}(\log)\to \dots\to \omega_B(E)^{\otimes i}\otimes R^if_*\Omega^{n-i}_{X/B}(\log).
	\end{equation*}
	We call $\varphi_i\colon f_*\omega_{X/B}\otimes T_B(-E)^{\otimes i}\to R^if_*\Omega^{n-i}_{X/B}(\log)$ the morphism obtained from $\phi_i$ by tensoring with $T_B(-E)^{\otimes i}$.
\end{defn} These morphisms are given on smooth fibers by iterated cup product with the Kodaira-Spencer class. 
From the Introduction, we recall that if $\phi_n$ is not zero (equivalently $\varphi_n$ is not zero) then $f$ is strongly non isotrivial.
%\begin{defn}
%	A fibration $f\colon X\to B$ is strongly non isotrivial if  $\varphi_n$ is not the trivial homomorphism.
%%	a proper subbundle of $ f_*\omega_{X/B}$, i.e. $\ker \phi_n\subsetneq f_*\omega_{X/B}$. 
%Equivalently $\phi_n$ is not trivial.
%\end{defn}
By the Fujita decomposition (\ref{fujita}), $	f_*\omega_{X/B}=\sU\oplus\sA$ and, since $\sU$ is contained in $K^0=\ker \phi_1$, we look at the image of $\sA$. 

In particular note that if $\sA$ is not contained in $K^0$ (i.e. $\sA\neq0$), and also $\phi_i(\sA)$ is not contained in $K^i\otimes \omega_B(E)^{\otimes i}$, for $i=1,\dots, n-1$, then the image $\phi_n(\sA)$ is not zero and  $f$ is strongly non isotrivial. Actually, in our setting, we are dealing with primitive forms. In fact the holomorphic $(n,0)$-forms on the fibers are primitive and hence the images $\phi_i(\sA)$ are bundles of primitive classes, since cup product of a primitive form with the Kodaira-Spencer class is again primitive. 

\subsubsection{Curvature conditions for strong non isotriviality}
A possible condition that ensures that $\phi_i(\sA)$ is not contained in $K^i_{\text{prim}}\otimes \omega_B(E)^{\otimes i}$, $i=1,\dots, n-1$,
since $\sA$ is ample, is that $K^i_{\text{prim}}\otimes \omega_B(E)^{\otimes i}$ induces a metric with semi-negative curvature on $\phi_i(\sA)\cap (K^i_{\text{prim}}\otimes \omega_B(E)^{\otimes i})$.

We have seen in Section \ref{sez1} how to define a smooth metric on $\sP^{n-i,i}|_{B\setminus E}$. Given a metric on $\omega_{B}(E)$ of local weight $\psi$ and curvature denoted by $c(\psi)$, we consider the tensor metric on the bundle $\sP^{n-i,i}\otimes \omega_{B}(E)^{\otimes i}|_{B\setminus E}$.
The following result gives a curvature formula for this metric.
%the bundle $\sP^{n-i,i}\otimes \omega_{B}(E)^{\otimes i}$ given a metric on $\omega_{B}(E)$ of local weight $\psi$ and curvature denoted by $c(\psi)$. 
We write an element of $\sP^{n-i,i}\otimes \omega_{B}(E)^{\otimes i}$ at the point $t$ as $[u_t]\otimes l_t$ where $l_t$ is of norm 1 and $u_t$ is  harmonic on $X_t$. 
\begin{prop}
The curvature $\Theta_i$ of $\sP^{n-i,i}\otimes \omega_{B}(E)^{\otimes i}|_{B\setminus E}$ satisfies 
\begin{equation}\label{tensor}
\langle\Theta_i ([u_t]\otimes l_t),[u_t]\otimes l_t\rangle=\lVert\eta_{t,h}\lVert^2-\lVert\zeta_{t,h}\lVert^2+ic(\psi)\lVert [u_t]\lVert^2
\end{equation}
where $\eta_{t,h}$, $\zeta_{t,h}$ are defined as in (\ref{curvatura3}).
\end{prop}
\begin{proof}
This is the formula for the curvature of a tensor product. Call for simplicity $F=\sP^{n-i,i}$ and $L= \omega_{B}(E)^{\otimes i}$. Then 
$\Theta_{F\otimes L}=\Theta_F\otimes I_L+I_F\otimes\Theta_L$ and 
$$
\langle\Theta_{F\otimes L} [u_t]\otimes l_t,[u_t]\otimes l_t\rangle=\langle\Theta_F [u_t],[u_t]\rangle\lVert l_t\lVert^2+\langle\Theta_L l_t,l_t\rangle\lVert [u_t]\lVert^2
$$ and we are done using that $\lVert l_t\lVert=1$ and Equation (\ref{curvatura3}). 
\end{proof}

We prove
\begin{prop}\label{negsing}
Assume that there exists a smooth metric on $\omega_B(E)$ of local weight $\psi$ such that on $B\setminus E$ we have $\lVert(\bar{k}_t\cup u_t)_h\lVert^2\geq ic(\psi)\lVert [u_t]\lVert^2$ for every section $[u]$ of $\varphi^i(\sA\otimes T_B(-E)^{\otimes i})\cap K^i_{\textnormal{prim}}$. Then $\phi_i(\sA)\cap (K^i_{\textnormal{prim}}\otimes \omega_B(E)^{\otimes i})$ admits a semi-negative singular Hermitian metric.
\end{prop}
\begin{proof}
	
	As recalled in Section \ref{sez1}, by \cite{BPW}, $K^i_{\text{prim}}$ has  a semi-negatively curved singular Hermitian metric. We consider on $K^{i}_{\text{prim}}\otimes \omega_B(E)^{\otimes i}$ the singular metric given by the tensor product and on  $\phi_i(\sA)\cap (K^i_{\text{prim}}\otimes \omega_B(E)^{\otimes i})$ its restriction.
	
	On $B\setminus E$ these metrics are smooth and coincide 
	with the restriction of the metric on $\sP^{n-i,i}\otimes \omega_{B}(E)^{\otimes i}|_{B\setminus E}$ discussed above.

For the sections of $K^{i}_{\text{prim}}\otimes \omega_B(E)^{\otimes i}$ we have by (\ref{tensor}) and (\ref{zeta})
\begin{equation}\label{curvaturau}
	\langle\Theta_i ([u_t]\otimes l_t),[u_t]\otimes l_t\rangle=-\lVert\zeta_{t,h}\lVert^2+ic(\psi)\lVert [u_t]\lVert^2=-\lVert(\bar{k}_t\cup u_t)_h\lVert^2+ic(\psi)\lVert [u_t]\lVert^2.
\end{equation}

 By our hypothesis, the metric on $\phi_i(\sA)\cap (K^i_{\text{prim}}\otimes \omega_B(E)^{\otimes i})$ is then semi-negatively curved on $B\setminus E$, since curvature decreases in subbundles.
 
 By the results in \cite{BPW}, it is not difficult to see that it is semi-negative  everywhere in the sense of singular metrics.
\end{proof}
%To sum up we have
%\begin{prop}
%If $\lVert(\bar{k}_t\cup u_t)_h\lVert^2\geq jc(\phi)\lVert [u_t]\lVert^2$  for all $[u_t]$ in $K^j_{\textnormal{prim}}$, $j=1,\dots,i$, then $K_{i}$  $\ker\phi_i$ is semi-negatively curved.
%\end{prop}
%\subsection{Strognly non isotrivial}
%The semi-negativity of $\ker \phi_i$  gives the following result on strongly non isotrivial fibrations. 

%First recall the following definition
%\begin{defn}
%A fibration $f\colon X\to B$ is strongly non isotrivial if the inclusion $\ker \phi_n\subset f_*\omega_{X/B}$ is strict.
%\end{defn}
So, under the hypotheses of Proposition \ref{negsing}, the image $\phi_i(\sA)$ is not contained in $K^i_{\text{prim}}\otimes \omega_B(E)^{\otimes i}$, $i=1,\dots, n-1$,
since $\sA$ is ample. Hence we conclude:
\begin{thm}\label{ineqk}
Let $f\colon X\to B$ be a semistable fibration and assume that there exists a smooth metric on $\omega_B(E)$ of local weight $\psi$ such that on $B\setminus E$ we have $\lVert(\bar{k}_t\cup u_t)_h\lVert^2\geq ic(\psi)\lVert [u_t]\lVert^2$  for every section  $[u]$ of $\varphi^i(\sA\otimes T_B(-E)^{\otimes i})\cap K^i_{\textnormal{prim}}$, $i=1,\dots,n-1$. Then either $f_*\omega_{X/B}$ is unitary flat or the fibration is strongly non isotrivial.
\end{thm}

%\begin{proof}
%By the previous proposition we have that $\ker \phi_n$ is semi-negatively curved while $f_*\omega_{X/B}$ is semi-positively curved. Hence either they are bot unitary flat, $f_*\omega_{X/B}=\ker \phi_n=\sU$ or the inclusion $\ker \phi_n\subset f_*\omega_{X/B}$ is strict and the fibration is strongly non isotrivial.
%\end{proof}

Note that a parallel approach is given by the study of the semi-positivity of $\sA\otimes T_B(-E)^{\otimes i}$ for $i=1,\dots,n-1$. Similarly as before, this condition shows that $\varphi_i (\sA\otimes T_B(-E)^{\otimes i})$ is not contained in   $K^i_{\text{prim}}$ for $i=1,\dots,n-1$ and hence $\varphi_n$ is not zero and $f$ is strongly non isotrivial (or $\sA$ is zero). We skip the computations since they are similar to the previous ones. We note that we only need the semi-positivity of $\sA\otimes T_B(-E)^{\otimes n-1}$ since it implies the semi-positivity of $\sA\otimes T_B(-E)^{\otimes i}$, $i=1,\dots,n-2$. This semi-positivity is given on $B\setminus E$ by the inequality $\lVert({k}_t\cup u_t)_h\lVert^2\geq (n-1)c(\psi)\lVert u_t\lVert^2$.
\begin{thm}\label{ineqa}
	Let $f\colon X\to B$ be a semistable fibration on a smooth projective curve $B$ and assume that there exists a smooth metric on $\omega_B(E)$ of local weight $\psi$ such that on $B\setminus E$ we have $\lVert({k}_t\cup u_t)_h\lVert^2\geq (n-1)c(\psi)\lVert u_t\lVert^2$  for all $u$ in $\sA$. Then either $f_*\omega_{X/B}$ is unitary flat or the fibration is strongly non isotrivial.
\end{thm}
\begin{rmk}
In principle there is no correlation between the inequalities of Theorem \ref{ineqk} and \ref{ineqa}, except when $n=2$. In fact in this case note that the condition $\lVert({k}_t\cup u_t)_h\lVert^2\geq c(\psi)\lVert u_t\lVert^2$ of Theorem \ref{ineqa} together with the inequality $\lVert(\bar{k}_t\cup({k}_t\cup u_t)_h)_h\lVert^2\leq \lVert({k}_t\cup u_t)_h\lVert \lVert u_t\lVert$ coming from Cauchy formula together with the fact that $k_t$ and $\bar{k_t}$ are adjoint (see Proposition \ref{adjoint} below), give the condition $\lVert(\bar{k}_t\cup({k}_t\cup u_t)_h)_h\lVert^2\geq c(\psi)\lVert({k}_t\cup u_t)_h\lVert^2$. This is the condition required in Theorem \ref{ineqk} since ${k}_t\cup u_t$ gives, while $u_t$ varies in $\sA$, all the elements in $\varphi^i(\sA\otimes T_B(-E)^{\otimes i})\cap K^i_{\text{prim}}$.
\end{rmk}
As a final note, everything in this section can actually be generalised for fibrations over a higher dimensional base. In fact, even when $B$ is not a curve, by \cite{CK} the direct image $f_*\omega_{X/B}$ has a decomposition $f_*\omega_{X/B}=\sU\oplus\hat{\sA}$ where $\sU$ is unitary flat and  $\hat{\sA}$ is generically ample, that is $\hat{\sA}$ restricts to an ample vector bundle on the general complete intersection smooth curve in $B$.
Hence it is possible to consider the generalisation of the morphism $\phi_i$ and $\varphi_i$ as
$\phi_i\colon f_*\omega_{X/B}\to  \text{Sym}^i\Omega^1_B(\log E)\otimes R^if_*\Omega^{n-i}_{X/B}(\log)$ and $\varphi_i\colon f_*\omega_{X/B}\otimes   \text{Sym}^iT_B(-\log E)\to R^if_*\Omega^{n-i}_{X/B}(\log)$ and argue similarly to obtain similar conditions involving the curvature of the symmetric bundles.
\subsubsection{A case on surfaces}
Now, consider again the composition $\phi_n$
\begin{equation*}
\phi_n\colon f_*\omega_{X/B}\to  \omega_B(E)\otimes R^1f_*\Omega^{n-1}_{X/B}(\log)\to\dots\to \omega_B(E)^{\otimes n}\otimes R^nf_*\sO_{X}
\end{equation*} 
and restrict it to the general smooth fiber $X_t$ as
\begin{equation}
\phi_{n,t}\colon H^{n,0}(X_t)\to  T_{B,t}^\vee\otimes H^{n-1,1}(X_t)\to\dots\to (T_{B,t}^\vee)^{\otimes n}\otimes H^{0,n}(X_t)
\end{equation}
Since $\dim B=1$, choosing a local coordinate around $t$ we identify $T_{B,t}^\vee\cong \mC$ and, with a little abuse of notation, we denote by $\xi_t\cup$ every map
$$
\xi_t\cup \colon H^{p,q}(X_t)\to H^{p-1,q+1}(X_t)
$$ since they are all defined by the cup product with the Kodaira-Spencer class $\xi_t$.
\begin{prop}\label{adjoint}
Let $[u]\in H^{p,q}_{\textnormal{prim}}(X_t)$ and $[v]\in H^{p-1,q+1}_{\textnormal{prim}}(X_t)$, $p+q=n$.  Then 
$$
\langle\xi_t\cup [u],[v]\rangle=\langle [u],\overline{\xi_t\cup \overline{[v]}}\rangle
$$ 
\end{prop}
%Note that the first member is computed in $H^{p-1,q+1}(X_t)$ and the second in  $H^{p,q}(X_t)$
\begin{proof}
We choose harmonic representatives $u_h$ and $v_h$ of the classes $[u]$ and $[v]$ respectively. Also we represent the class $\xi_t$ by its representative $k_t$ as above.

We have
$$
\langle\xi_t\cup [u],[v]\rangle=(-1)^{q+1}c_n\int_{X_t}(\xi_t\cup [u])_h\wedge\bar{v}_h=(-1)^{q+1}c_n\int_{X_t}(k_t\cup u_h)_h\wedge\bar{v}_h=(-1)^{q+1}c_n\int_{X_t}( k_t\cup u_h)\wedge\bar{v}_h
$$ and
$$
\langle [u],\overline{\xi_t\cup \overline{[v]}}\rangle=(-1)^{q}c_n\int_{X_t} u_h\wedge(
{\xi_t\cup \overline{[v]}})_h=(-1)^{q}c_n\int_{X_t} u_h\wedge({k_t\cup \bar{v}_h})_h=(-1)^{q}c_n\int_{X_t} u_h\wedge({k_t\cup \bar{v}_h})
$$
It is now easy to see, for example writing locally $u_h,v_h$ and $k_t$, that $(-1)^{q+1}c_n\int_{X_t}( k_t\cup u_h)\wedge\bar{v}_h=(-1)^{q}c_n\int_{X_t} u_h\wedge({k_t\cup \bar{v}_h})$. This can for example be verified on decomposable forms, since everything is linear. 
\end{proof}
\begin{rmk}
	From this formula we easily get the condition
 $K^q_{\textnormal{prim},t}\subset \overline{\xi_t(H^{q+1,p-1}_{\textnormal{prim}}(X_t))}^\perp$.
\end{rmk}
Denote by $\xi^q_t$ the composition of the cup product with $\xi_t$ taken $q$ times,
$$
\xi_t^q\colon H^{n,0}(X_t)\to H^{n-q,q}(X_t)
$$ 
Hence we have the following condition for strong non isotriviality
\begin{prop}\label{inters}
If $\xi_t^{n-1}(H^{n,0}(X_t))\cap \overline{\xi_t(H^{n,0}(X_t))}\neq \{0\}$ then the fibration is strongly non isotrivial.
\end{prop}
\begin{proof}
This comes from the fact that all the elements of $K^{n-1}_{\textnormal{prim},t}$ are contained in $\overline{\xi_t(H^{n,0}_{\textnormal{prim}}(X_t))}^\perp$ by the previous remark. Hence $\phi_{n,t}=\xi^n_t\neq 0$.
\end{proof}

In the case of $n=2$, this condition  suggests an interesting numerical bound concerning families of surfaces with positive index.
\begin{thm}\label{indicepositivo}
Let $f\colon X\to B$ be a semistable fibration and denote by $r$ the rank of $K^0$. Assume that the general fiber $X_t$ is a surface satisfying $K^2_{X_t}> 8\chi(\sO_{X_t})+2r+1$, then $f$ is strongly non isotrivial.
\end{thm}
\begin{proof}
By Proposition \ref{inters}, it is enough to show that the intersection $\xi_t(H^{2,0}(X_t))\cap \overline{\xi_t(H^{2,0}(X_t))}$ is nontrivial in $H^{1,1}_{\text{prim}}(X_t)$. 
%This is true for example if $\dim \xi(H^{2,0}(S))>h^{1,1}/2$, that is $p_g-r>h^{1,1}/2$. Now since $c_2=2-4q+2p_g+h^{1,1}$ and by Noether's formula $12\chi(S)=K_S^2+c_2$, our condition becomes 
%$$
%2p_g-2r>-2p_g+4q-2+12\chi(S)-K_S^2
%$$ which we can write as 
%$$
%K^2_S> 8\chi(\sO_S)+2(r+1).
%$$
%VERSIONE 2
This is true if $\dim \xi_t(H^{2,0}(X_t))>h^{1,1}_{\text{prim}}/2$, that is $p_g-r>h^{1,1}_{\text{prim}}/2$. Now since $c_2=2-4q+2p_g+h^{1,1}$ and $h^{1,1}_{\text{prim}}\leq h^{1,1}-1$, by Noether's formula $12\chi(\sO_{X_t})=K_{X_t}^2+c_2$, the above condition is implied by 
$$
2p_g-2r>-2p_g+4q-2+12\chi(\sO_{X_t})-K_{X_t}^2-1
$$ which we can write as 
$$
K^2_{X_t}> 8\chi(\sO_{X_t})+2r+1.
$$
\end{proof}
\begin{rmk}\label{upperb}
	Note that since $9\chi(\sO_{X_t})\leq K^2_{X_t}$ by the  Miyaoka-Yau inequality (\cite{Mi,Y}), our assumption can hold only if $r\leq\frac{p_g-q}{2}$.
\end{rmk}

\subsection{Conditions for triviality of a fibration}

In this last section we look at the opposite problem and give conditions for the triviality of $f\colon X\to B$ over a Zariski open subset of $B$.  So far we have considered the direct image of the relative dualizing sheaf $f_*\omega_{X/B}$, here we will work on the direct image of the relative pluricanonical bundles $f_*(\omega_{X/B}^{\otimes m})$. Hence we briefly recall how to construct a Hermitian metric on the vector bundle associated to $f_*(\omega_{X/B}\otimes L)$, where $L$ is a line bundle on $X$ endowed with a singular metric assumed to be smooth when restricted to the general fiber. We denote by $\psi$ the local weight of this metric and assume that $L$ is semi-positively curved.

Without covering all the details, we just highlight the differences with the untwisted case of Section \ref{sez1}. In particular, on an appropriate open subset of $B$, the metric \ref{norma1} is replaced by
\begin{equation}
	\label{norma3}
	\lVert u_t\lVert^2_t=\int_{X_t} c_n u_t\wedge\bar{u}_te^{-\psi}
\end{equation}
and the curvature Formula \ref{curvatura1} is replaced by 
\begin{equation}
	\label{curvaturatwist}
	\langle\Theta u_t,u_t \rangle_t=f_*(c_ni\partial\bar{\partial}\psi\wedge\bu\wedge\bar{\bu}e^{-\psi})/dV_t+\lVert\eta_t\lVert^2.
\end{equation}  where  $dV_t=idt\wedge d\bar{t}$.

For our purposes we take $L=\omega_{X/B}^{\otimes {m-1}}$ and we prove the following theorem which complements the constancy result proved in \cite{XY}.

\begin{thm}\label{canpol}
	Let $f\colon X\to B$ a family of canonically polarised manifolds. Assume that there exists a surjective morphism $\rho\colon V\times B\to X$, where $V$ is a projective  variety, such that $p:=f\circ \rho$ is the projection on $B$. Then $X$ is Zariski locally trivial.
\end{thm}
\begin{proof}
	By \cite{I,LS}, the direct image $f_*(\omega_{X/B}^{\otimes m})$ has a so called Catanese–Fujita–Kawamata decomposition
	$$
	f_*(\omega_{X/B}^{\otimes m})=\sU\oplus \sA 
	$$ where $\sU$ is unitary flat and $\sA$ is ample (or zero). On the other hand, we also get an injective morphism
	 $f_*(\omega_{X/B}^{\otimes m})\hookrightarrow p_*((\Omega^n_{V\times B/B})^{\otimes m})=H^0(V, (\Omega^n_V)^{\otimes m})\otimes \sO_B$.
	 From this, it is not difficult to see that $\sA=0$ and $\sU=\sO_B^l$ is trivial.
	 
	 Now we write $f_*(\omega_{X/B}^{\otimes m})=f_*(\omega_{X/B}\otimes\omega_{X/B}^{\otimes m-1})$ and take $L=\omega_{X/B}^{\otimes m-1}$.
	 $L$ admits a singular metric which is smooth and with strictly positive curvature when restricted to the smooth fibers; cf. \cite{Sc}. By \cite[Theorem 3.3.5]{PaTa},  the vector bundle $f_*(\omega_{X/B}^{\otimes m})$ admits a singular Hermitian metric with semi-positive curvature.  But by the fact that $f_*(\omega_{X/B}^{\otimes m})$ is trivial, the curvature of this metric cannot be strictly positive definite; actually it is zero, see for example \cite[Lemma 13.2]{HPS}. On the open subset where $f$ is smooth, this metric coincides with the one recalled above in (\ref{norma3}) and by the explicit curvature formula given in \cite[Theorem 1.2]{bo2} and following remarks, the Kodaira-Spencer class vanishes at every point of this open subset. 
	 
	 Moreover, by the argument in \cite[Section 4.1 page 1216]{bo2}, the  Kodaira-Spencer class $\xi_t$ is represented by the form $\bar{\partial}{V_\psi}_{|X_t}$ where $V_\psi$ is a (in principle only smooth) vector field which turns out to be holomorphic since the curvature of $f_*(\omega_{X/B}^{\otimes m})$ is not strictly positive. The fibration $f$ is then analytically locally trivial.
	 
	 To show that $f$ is trivial over a Zariski open subset $U$ of $B$ we can proceed then as in \cite[Page 22]{XY}.
\end{proof}


\begin{thebibliography}{Muk04}




\bibitem[B1]{bo} B. Berndtsson, \emph{Curvature of vector bundles associated to holomorphic fibrations}, Ann. of Math. 169 (2009), no. 2, 531--560.

\bibitem[B2]{bo2} B. Berndtsson, \emph{Strict and nonstrict positivity of direct image bundles}, Math. Z. 269 (2011), no. 3-4, 1201--1218.

 

\bibitem[BPW]{BPW} B. Berndtsson, M. P\u{a}un, X. Wang, \emph{Algebraic fiber spaces and curvature of higher direct images}, J. Inst. Math. Jussieu 21 (2022), no. 3, 973--1028.
%\bibitem[BHT]{BHT}
%F. Bogomolov, B. Hassett, and Y. Tschinkel, editors.
%\newblock {\em Birational {Geometry}, {Rational} {Curves}, and {Arithmetic}}.
%\newblock Simons {Symposia}. Springer--Verlag, New York, 2013.
%
%\bibitem[BL]{BL}C. Birkenhake, H. Lange \emph{Complex Abelian Varieties} Grundlehren der
%mathematischen Wissenschaften 302, A Series of Comprehensive Studies in Mathematics, Springer-Verlag Berlin Heidelberg  (1992).
%C. Voisin, \emph{Hodge theory and complex algebraic geometry, I}. Translated from the French by Leila Schneps. Cambridge Studies in Advanced Mathematics, 76. Cambridge University Press, Cambridge, 2002.
    
%\bibitem[An]{andr} A. Andreotti, {\em  On a theorem of Torelli}, Amer. J. Math. 80 (1958), 801--828.
%
%
%\bibitem[BC]{BC} I. Bauer, F. Catanese, {\em Symmetry and variations of Hodge Structures}, Asian J. Math. 8 (2004), no. 2, 363--390. 
%
%\bibitem[Ca1]{Ca1}
%F. Catanese, {\em Infinitesimal Torelli theorems and counterexamples to
%Torelli problems}, Topics in transcendental algebraic geometry
%(Princeton, N.J., 1981/1982), 143--156, Ann. of Math. Stud., 106, Princeton Univ. Press,
%Princeton, NJ, (1984).
%
%\bibitem[C]{Ca2} F. Catanese, \emph{Moduli
%and classification of irregular Kaehler manifolds (and algebraic varieties) with Albanese
%general type fibrations},  Invent. Math. 104 (1991), no. 2, 263--289. 

\bibitem[CD1]{CD1}
F. Catanese, M. Dettweiler, \emph{Answer to a question by Fujita on Variation of Hodge Structures}, Higher Dimensional Algebraic Geometry: In honour of Professor Yujiro Kawamata's sixtieth birthday, Mathematical Society of Japan, Tokyo, (2017), 73--102.
\bibitem[CD2]{CD2}
F. Catanese, M. Dettweiler,  \emph{The direct image of the relative dualizing sheaf needs not be semiample}, C. R. Math. Acad. Sci. Paris 352 (2014), no. 3, 241--244.
%\bibitem[CD3]{CD3}
%F. Catanese, M. Dettweiler, \emph{Vector bundles on curves coming from variation of Hodge structures}, Internat. J. Math. 27 (2016), no. 7, 1640001, 25 pp.
\bibitem[CK]{CK}
F. Catanese, Y. Kawamata, \emph{Fujita decomposition over higher dimensional base}, Eur. J. Math. 5 (2019), no. 3, 720--728.





%\bibitem[CCM]{CCM}
%F. Catanese, C. Ciliberto, M. Mendes Lopes,\emph{ On the 
%classification of irregular surfaces of general type with nonbirational bicanonical map},
%Trans. Amer. Math. Soc. 350 (1998), no. 1, 275--308. 
%
%\bibitem[CS]{CS} F. Catanese, F-O Schreyer,  \emph{Canonical 
%projections of irregular algebraic surfaces},
%Algebraic geometry, 79�116, de Gruyter, Berlin, (2002). 
%
%
%\bibitem[CMP]{CMP} J. Carlson, S. M{\"{u}}ller-Stach, C. Peters, 
%{\em Period Mappings and Period Domains}, Cambridge Studies in 
%Advanced Mathematics, 85. Cambridge University Press, Cambridge, (2003).
%

%\bibitem[CEZGT]{CEZGT} Eduardo Cattani, Fouad El Zein, Phillip A. Griffiths, and L\^e D\~ung Tr\'ang, editors.
%{\em Hodge theory}, volume 49 of Mathematical Notes. Princeton University Press, Princeton,
%NJ, 2014.

%\bibitem[CMR]{CMR}
% C. Ciliberto, M. Mendes Lopes, X. Roulleau \emph{ On Schoen surfaces},
%Comment. Math. Helv. 90 (2015), 59-74. 

%  


%\bibitem[Co]{cox}
%D. A. Cox, {\em Generic Torelli and infinitesimal variation of Hodge 
%structure}, Algebraic geometry, Bowdoin, 1985 (Brunswick, Maine, 1985), 235--246, Proc. 
%Sympos. Pure Math., 46, Part 2, Amer. Math. Soc., Providence, 
%RI, (1987).
%
%
%\bibitem[Cod]{Cod}
%G. Codogni, {\em Satake compactifications, Lattices and Schottky problem},  Ph.D. Thesis, (2013).
%

%
%\bibitem[C]{C}
%G. Ceresa, \emph{C is not algebraically equivalent to $C^-$ in its Jacobian}, Ann. of Math. (2) 117 (1983), 285--291.

\bibitem[CP]{CP} A. Collino, G.P. Pirola, 
\emph{The Griffiths infinitesimal invariant for a curve in its Jacobian}, Duke Math. J. 78 (1995), no. 1, 59--88.


%
%
%\bibitem[De]{De}
%P. Deligne, \emph{Equations Diff\'erentielles \`a Points Siunguliers R\'eguliers}, LNM 163, Springer-Verlag ,Berlin, Heidelberg, New York,  1970. 
%
%\bibitem[Dr]{Dr}
%S. Druel. \emph{Codimension 1 foliations with numerically trivial canonical class on singular spaces.} Duke Math. J. 170 (1) 95--203, (2021). 


%\bibitem[Fa]{Fa}
%N. Fakhruddin, \emph{Algebraic cycles on generic Abelian 
%varieties,} Comp. Math., 100, (1996), 101--119.


\bibitem[F1]{Fu}
T. Fujita, \emph{On K\"ahler fiber spaces over curves}, 
 J. Math. Soc. Japan 30 (1978), no. 4, 779--794. 

\bibitem[F2]{Fu2}
T. Fujita, \emph{The sheaf of relative canonical forms of a K\"ahler fiber
space over a curve}, Proc. Japan Acad. Ser. A Math. Sci. 54 (1978),
no. 7, 183--184.
%
%\bibitem[G-A1]{victor1}
%V. Gonz\'alez-Alonso, \emph{Hodge numbers of irregular varieties and fibrations}, Ph.D. Thesis, (2013).

%\bibitem[Fl]{flenner}
%H. Flenner, {\em The infinitesimal Torelli problem for zero sets of sections of
%vector bundles},   Math. Z. 193 (1986), no. 2, 307--322. 
%
%\bibitem[G-A1]{victor1}
%V. Gonz\'alez-Alonso, \emph{Hodge numbers of irregular varieties and fibrations}, Ph.D. Thesis, (2013).
%
%%
\bibitem[GST]{GST}
V. Gonz\'alez-Alonso, L. Stoppino and S. Torelli, \emph{On the rank of the flat unitary summand of the Hodge bundle}, Trans. Amer. Math. Soc. 372 (2019), no. 12, 8663--8677.

\bibitem[GT]{GT}
V. Gonz\'alez-Alonso, S. Torelli, \emph{Families of curves with Higgs field of arbitrarily large kernel}, Bull.
London Math. Soc. 53 (2021), no. 2, 493--506.
%
\bibitem[G1]{G} P. Griffiths, \emph{Hermitian differential geometry, Chern classes, and positive vector bundles}, Global Analysis (Papers in Honor of K. Kodaira), Univ. Tokyo Press, Tokyo, (1969), 185--251.
\bibitem[G2]{G2} P. Griffiths, \emph{Periods of integrals on algebraic manifolds. III. Some global differential-geometric properties of the period mapping}, Inst. Hautes \'Etudes Sci. Publ. Math. No. 38 (1970), 125--180.

\bibitem[GH]{GH} P. Griffiths, J. Harris, \emph{Principles of algebraic geometry}, Pure and Applied Mathematics, Wiley-Interscience, New York, 1978.
\bibitem[GrT]{gt}
P. Griffiths, L. Tu, \emph{Curvature properties of the Hodge bundles}, Topics in transcendental algebraic geometry (Princeton, N.J., 1981/1982), 29--49, Ann. of Math. Stud., 106, Princeton Univ. Press, Princeton, NJ, 1984. 



\bibitem[HPS]{HPS} C. Hacon, M. Popa, C. Schnell,  \emph{Algebraic fiber spaces over abelian varieties: around a recent theorem by Cao and P\u aun}, Local and global methods in algebraic geometry, 143--195, Contemp. Math., 712, Amer. Math. Soc., Providence, RI, (2018).

\bibitem[In]{In} M. Inoue \emph{ Some New Surfaces of General Type}, Tokyo J. Math. Vol. 17, No. 2 (1994), 295--319.


\bibitem[Iw]{I} M. Iwai, \emph{Almost nef regular foliations and Fujita's decomposition of reflexive sheaves}, Ann. Sc. Norm. Super. Pisa Cl. Sci. (5) 23 (2022), no. 2, 719--743.

\bibitem[KKMSD]{KKMSD} G. Kempf, F.F. Knudsen, D. Mumford,  B. Saint-Donat, \emph{Toroidal embeddings. I}, Lecture Notes in Mathematics, Vol. 339. Springer-Verlag, Berlin-New York, 1973.

\bibitem[Kob]{Kob}
S. Kobayashi, \emph{Differential geometry of complex vector bundles}, Publications of the Mathematical Society of Japan, 15. Kan\^o Memorial Lectures, 5. Princeton University Press, Princeton, NJ, 1987. 
\bibitem[Kol]{K} J. Koll\'ar, \emph{Higher direct images of dualizing sheaves. II}, Ann. of Math. 124 (1986), no. 1, 171--202.
\bibitem[Kov]{Kov} S. Kov\'acs, \emph{Strong non-isotriviality and rigidity}, Recent progress in arithmetic and algebraic geometry, 47--55, Contemp. Math., 386, Amer. Math. Soc., Providence, RI, 2005.
\bibitem[LS]{LS} L. Lombardi, C. Schnell, \emph{Singular hermitian metrics and the decomposition theorem of Catanese, Fujita, and Kawamata}, Proc. Amer. Math. Soc. 152 (2024), no. 1, 137--146.

\bibitem[LY]{LY} S.S. Lu, S.T. Yau, \emph{ Holomorphic curves in surfaces of general type}, 
Proc. Natl. Acad. Sci. USA (1990) 80--82.

\bibitem[Mi]{Mi}
Y. Miyaoka, \emph{On the Chern numbers of surfaces of general type}, Invent. Math. 42 (1977), 225--237.



\bibitem[P]{P} M. P\u{a}un, \emph{Singular Hermitian metrics and positivity of direct images of pluricanonical bundles}, Algebraic geometry: Salt Lake City 2015, 519--553, Proc. Sympos. Pure Math., 97.1, Amer. Math. Soc., Providence, RI, 2018. 



\bibitem[PaTa]{PaTa}
M. P\u aun, S. Takayama, \emph{Positivity of twisted relative pluricanonical bundles and their direct images}, J. Algebraic Geom. 27 (2018), no. 2, 211--272.
\bibitem[PT]{PT}
G.P. Pirola, S. Torelli, \emph{Massey Products and Fujita decompositions on fibrations of curves}, Collect. Math. 71 (2020), no. 1, 39--61.
%
%\bibitem[PS]{PS} C. A. M. Peters, J. H. M. Steenbrink, {\em Mixed Hodge structures}, Volume 52
%of Ergebnisse der Mathematik und ihrer Grenzgebiete. 3. Folge. A Series of Modern
%Surveys in Mathematics. Springer--Verlag, Berlin, 2008.
%
%


%\bibitem[P]{P}
%G. P. Pirola, \emph{On a conjecture of Xiao,} J. Reine Angew. Math., 431 (1992), 75--89.
%
\bibitem[PZ]{PZ}
G.P. Pirola, F. Zucconi, \emph{Variations of the Albanese morphisms,} J. Algebraic Geom. 12 (2003), no. 3, 535--572. 



\bibitem[Rau]{raufi} H. Raufi, \emph{Singular hermitian metrics on holomorphic vector bundles}, Ark. Mat. 53 (2015), no. 2, 359--382.
%
%\bibitem[Rav]{Ra}
%E. Raviolo, {\em Some geometric applications of the theory of variations of Hodge structures}, Ph.D. Thesis.




\bibitem[R]{R} L. Rizzi, \emph{Massey products and Fujita decomposition over higher dimensional base},
Collect. Math. 75 (2024), no. 2, 511--534.


\bibitem[RZ1]{RZ1} L. Rizzi, F. Zucconi, {\em Differential forms and quadrics of the canonical image}, Ann. Mat. Pura Appl. 199 (2020), no. 6, 2341--2356.

 
\bibitem[RZ2]{RZ2} L. Rizzi, F. Zucconi, {\em Generalized adjoint forms on algebraic varieties,}  Ann. Mat. Pura Appl. 196 (2017), no. 3, 819--836.
\bibitem[RZ3]{RZ3} L. Rizzi, F. Zucconi, {\em On Green’s proof of infinitesimal Torelli theorem for hypersurfaces,} Atti Accad. Naz. Lincei Rend. Lincei Cl. Sci. Fis. Mat. Natur. 29 (2018),
no. 4, 689--709.


\bibitem[RZ4]{RZ4} L. Rizzi, F. Zucconi, {\em Fujita decomposition and Massey product for fibered varieties}, Nagoya Math. J. 247 (2022), 624--652.

\bibitem[RZ5]{RZ5} L. Rizzi, F. Zucconi, \emph{Local systems of Castelnuovo-type and factorization of semistable families}, J. Geom. Phys. 211 (2025).

%\bibitem[RZ6]{RZ6} L. Rizzi, F. Zucconi, \emph{Massey products and Fujita decomposition}, Angles of Geometry: The Nottingham Algebraic Geometry Seminar, L. Campo, J. Hofscheier, A. Kasprzyk (eds.), World Scientific, in press, 2024.
\bibitem[Sc]{Sc} G. Schumacher, \emph{Positivity of relative canonical bundles and applications}, Invent. Math. 190 (2012), no. 1, 1--56. 
\bibitem[Siu]{Siu} Y.T. Siu, \emph{Curvature of the Weil-Petersson metric in the moduli space of compact K\"ahler-Einstein manifolds of negative first Chern class}, Contributions to several complex variables, 261--298, Aspects Math., E9, Friedr. Vieweg, Braunschweig, 1986.
\bibitem[St]{S}
J.H.M. Steenbrink, \emph{Mixed Hodge structure on the vanishing cohomology}, Real and complex singularities (Proc. Ninth Nordic Summer School/NAVF Sympos. Math., Oslo, 1976), pp. 525--563, Sijthoff and Noordhoff, Alphen aan den Rijn, 1977.
%\bibitem[To]{to}
%R. Torelli, {\em Sulle variet\`{a} di Jacobi, I, II}, Rendiconti R. 
%Accad. dei Lincei {22-2}  (1913), 98--103, 437--441.
%
%\bibitem[To]{To} S. Torelli, {\em Fujita decompositions and infinitesimal
%invariants on fibred surfaces}, 2018, PhD Thesis.

%\bibitem[Vo]{Vo} C. Voisin, {\it Une Remarque Sur l'Invariant Infinit\'esimal 
%Des Fonctions Normales,} C. R. Acad. Sci. Paris, t. 307, S\'erie I, 
%(1988), 157--160.
%
%
%\bibitem[U]{U}
%K. Ueno, \emph{Classification theory of algebraic varieties and compact complex spaces}, Lecture
%Notes in Math., vol. 439, Springer-Verlag, 1975.
%\bibitem[U]{U} H. Umemura, \emph{Some results in the theory of vector bundles}, Nagoya Math. J. 52 (1973), 97--128.

%\bibitem[Vo1]{Vo1}
%C. Voisin, \emph{Hodge theory and complex algebraic geometry, I}. Translated from the French by Leila Schneps. Cambridge Studies in Advanced Mathematics, 76. Cambridge University Press, Cambridge, 2002.
%
%\bibitem[Vo2]{Vo2}
%C. Voisin, \emph{Hodge theory and complex algebraic geometry, II}. Translated from the French by Leila Schneps. Cambridge Studies in Advanced Mathematics, 77. Cambridge University Press, Cambridge, 2003.

%\bibitem[We]{we}
%A. Weil, {\em
%Zum Beweis des Torellischen Satzes}, Nachr. Akad. Wiss. G\"ottingen, 
%Math.-Phys. Kl.   IIa (1957), 32--53.
%\bibitem[X]{X}
%G. Xiao, \emph{Fibered Algebraic Surfaces with Low Slope,} Math. Ann. 276,  (1987), 449--466.

%
\bibitem[XY]{XY}
J. Xie, X. Yuan, \emph{Partial Heights and the Geometric Bombieri-Lang Conjecture}, arXiv.2305.14789. 
\bibitem[Y]{Y}
S.T. Yau, \emph{On the Ricci curvature of a compact K\"ahler manifold and the complex Monge-Amp\`ere equation. I}, Comm. Pure Appl. Math. 31 (1978), no. 3, 339--411. 

\bibitem[Z]{Z} K. Zuo, \emph{On the negativity of kernels of Kodaira-Spencer maps on Hodge bundles and applications}, Kodaira's issue. Asian J. Math. 4 (2000), no. 1, 279--301.
\end{thebibliography}
\end{document}